\documentclass[a4paper,10pt]{article}
\usepackage[utf8]{inputenc}
\usepackage{geometry}

\usepackage{amsmath}
\usepackage{amssymb}
\usepackage{amsthm}
\usepackage[initials, shortalphabetic]{amsrefs}
\usepackage{tikz-cd}
\usepackage{hyperref}
\usepackage{cleveref}
\usepackage[aligntableaux=center]{ytableau}

\newcommand{\C}{\mathbb{C}}
\newcommand{\g}{\mathfrak{g}}
\newcommand{\gl}{\mathfrak{gl}}
\newcommand{\p}{\mathfrak{p}}
\newcommand{\kl}{\mathfrak{k}}
\newcommand{\sol}{\mathfrak{o}}
\newcommand{\spl}{\mathfrak{sp}}
\DeclareMathOperator{\rk}{rk}
\DeclareMathOperator{\codim}{codim}
\DeclareMathOperator{\Lie}{Lie}
\DeclareMathOperator{\Hom}{Hom}
\DeclareMathOperator{\GL}{GL}
\DeclareMathOperator{\aug}{aug}

\newcommand{\Qed}{}

\theoremstyle{plain}
\newtheorem{theorem}{Theorem}
\newtheorem{lemma}[theorem]{Lemma}
\newtheorem{proposition}[theorem]{Proposition}

\theoremstyle{definition}
\newtheorem{definition}[theorem]{Definition}

\theoremstyle{remark}
\newtheorem*{remark}{Remark}
\newtheorem*{example}{Example}

\title{Normality of closure of orthogonal nilpotent symmetric orbits}
\author{Marco Trevisiol}

\begin{document}
\maketitle

\begin{abstract}
    We study closures of conjugacy classes in the symmetric matrices of the orthogonal group and we determine which one are normal varieties. In contrast to the result for the symplectic group where all classes have normal closure, there is only a relatively small portion of classes with normal closure. We perform a combinatorial computation on top of the same methods used by Kraft-Procesi and Ohta.
\end{abstract}

\section*{Introduction}

In a fundamental paper of Kostant \cite{kostant1963lie}, the adjoint action on a reductive Lie algebra $\g$ defined over an algebrically closed field $k$ of characteristic 0 is studied in detail. In the course of his  analysis it arose the following problem. Let $A \in \g$, let $C_A$ be the conjugacy class of $A$ and let $\overline{C_A}$ be the (Zarisky) closure of $C_A$. Is $\overline{C_A}$ always a normal variety?

In his paper Kostant showed that if $A$ is a regular nilpotent element of $\g$, so that $\overline{C_A}$ is the nilpotent cone of $\g$, the normality is always the case.

In \cite{hesselink1976singularities}, the problem of normality of $\overline{C_A}$ is reduced to the case in which $A$ is nilpotent, possibly changing the Lie algebra $\g$.

Later, in \cite{broer1994normality}*{Theorem 4.1}, Broer proved the normality of a handful of nilpotent orbits in $\g$, including the regular and the subregular orbits (if $\g$ is simple).

In the fundamental paper \cite{kraft1979closures} of Kraft and Procesi, the normality of $\overline{C_A}$ is proved for every nilpotent $A$ in the case $\g = \gl_n$. Their method consists in constructing an auxiliary variety $Z$ which is a normal complete intersection such that $\overline{C_A}$ is a \emph{quotient} of $Z$.

Kraft and Procesi extended their method to the case where $\g$ is the orthogonal or symplectic Lie algebra in \cite{kraft1982geometry}. In that case not all nilpotent classes have normal closure. They obtained necessary and sufficient conditions on the partition of $A$ in order to have the normality of $\overline{C_A}$, under suitable hypothesis on $A$.

In a subsequent paper by Sommers \cite{sommers2005normality}, the cases outside of this hypothesis (in \cite{kraft1982geometry}) are proved to always have normal closure. Sommers also solved the problem for the exceptional case $E_6$ in \cite{sommers2003normality}.

Kostant and Rallis generalized the study  \cite{kostant1963lie} of the adjoint action on $\g$ to the action on the symmetric space in \cite{kostant1971orbits}. We briefly describe their setting. Let $\theta$ be a Lie algebra automorphism of $\g$ of order $2$. Then there exists a decomposition of $\g$ in $\theta$-eigenspaces given by
\[
    \g = \kl \oplus \p
\]
where $\kl = \{A \in \g : \theta(A) = A\}$ and $\p = \{A \in \g : \theta(A) = -A\}$. Let $G$ be the adjoint group of $\g$ and $K \subseteq G$ be the subgroup of elements commuting with $\theta$. We notice that $\kl$ is a Lie subalgebra of $\g$ and $\Lie K = \kl$. We remark that $K$ is not necessarily connected. Moreover the adjoint action of $G$ on $\g$ induces an action of $K$ on $\g$ and both $\kl$, $\p$ are $K$-stable. The pair $(\g, \kl)$ is called a symmetric pair and $\p$ is the symmetric space associate to it.

In their study, Kostant and Rallis focused on the action of $K$ on $\p$. Among other things, they already observed that the nilpotent cone is not irreducible nor normal in general.

In \cite{vinberg1976weyl}, Vinberg further generalized the study of the adjoint action to graded Lie algebras, also called $\theta$-groups. Here $\theta$ is an automorphism of a Lie algebra $\g$ of finite order $n$, thus generalizing the case $n=2$ of \cite{kostant1971orbits}. The graduation induced by $\theta$, $\g = \g_0 \oplus \dots \oplus \g_{n-1}$, produces an action of the $\theta$-fixed points of the adjoint group on $\g_1$.

After the work of  Kostant, Rallis and Vinberg, the problem of the normality of $\overline{C_A}$ naturally generalizes to the case of symmetric spaces.

In \cite{sekiguchi1984nilpotent}, Sekiguchi studied deeply the geometry of the nilpotent cone in symmetric spaces. In particular he proved several results for the principal nilpotent elements (which are the analogue of the regular nilpotent elements) when the symmetric space is the orthogonal symmetric space or the symplectic symmetric space, that is, the symmetric spaces of the symmetric pairs $(\g, \kl) = (\gl(n), \sol(n))$ and $(\g, \kl) = (\gl(2m), \spl(m))$ respectively.

The method of the auxiliary variety $Z$ developed by Kraft and Procesi in \cites{kraft1979closures, kraft1980minimal, kraft1982geometry} was adapted by Ohta in \cite{ohta1986singularities} to the study of the singularities of the orbits in the orthogonal and symplectic symmetric spaces. In particular he proved that $\overline{C_A}$ is always normal in the case of symplectic symmetric space. Moreover he found several orbits $C_A$ in the case of the orthogonal symmetric space which have a non-normal closure.

It is well known that the conjugacy classes of the nilpotent cone of the orthogonal symmetric space are parametrized by the partitions of $n$ (see \cref{sec:sno} for more details). The main purpose of this paper is to give a necessary and sufficient condition on the partition corresponding to the orbit of a nilpotent symmetric element $A$ in order to have the normality of $\overline{C_A}$. The following theorem is the main result of the paper.

\begin{theorem}
    \label{thm:intro:main}
    Let $\p$ be the symmetric space of the symmetric pair $(\gl(n),\sol(n))$ and let $A \in \p$ be a nilpotent element. Let $\lambda = (\lambda_1, \dots, \lambda_h)$ be the partition of $A$. Then $\overline{C_A}$ is normal if and only if
    \begin{equation}
        \label{eqn:intro:1step}
        \lambda_{i} - \lambda_{i+1} \leq 1 \quad \forall i = 1, \dots, h
    \end{equation}
    with the convention that $\lambda_{h+1} = 0$.
\end{theorem}

The proof of \cref{thm:intro:main} will be carried out in \cref{sec:nor}. The \emph{only if} part was already proved in \cites{sekiguchi1984nilpotent, ohta1986singularities} (see \cref{sec:nor} for details).

We remark that, if $n$ is even, $C_A$ is not necessarily connected, as the group $K$ is not connected. Therefore, sometimes, $\overline{C_A}$ is not normal just because it is not irreducible, being the union of two closure of $SO(n)$-orbits (as an example, if one takes $\lambda=(2)$, one can shows that $\overline{C_A}$ is the union of two lines which meet at $0$). As the referee has kindly pointed out, it would be interesting to investigate the normality of the closure of the $SO(n)$-orbits.

We summarise the content of the rest of the paper.

In \cref{sec:sno} we recall some basic facts about symmetric nilpotent orbits for the orthogonal group. In \cref{sec:ortoab} we recall from \cite{kraft1979closures} the classification of nilpotent pairs via $ab$-diagrams and we describe the class of $ab$-diagrams which parametrize symmetric nilpotent pairs. In \cref{sec:zvar} we recall the construction of the variety $Z$ and some of its properties which are needed for the proof. In \cref{sec:sstep} we define a condition on partitions which plays an important role in our investigation. In  \cref{sec:diff} we introduce some combinatorial description of pairs of partitions. In  \cref{sec:ci}, resp. \cref{sec:nor}, we prove a complete intersection, resp. normality, condition for the variety $Z$ using a combinatorial computation carried on in \cref{sec:comb}.

I am grateful to Claudio Procesi for proposing me this problem. I would like to thank Corrado De Concini, Andrea Maffei and Paolo Bravi for many useful comments. I am really grateful to the reviewers for their precise and useful comments. A special thank goes to my PhD advisor Giovanni Cerulli Irelli for many discussions about this problem and his precious help in editing this paper.

\section{Symmetric nilpotent orbits}
\label{sec:sno}
In this section we introduce settings and notations for the objects studied in the paper. We follow \cite{kraft1982geometry} and \cite{kostant1971orbits}.

Let $V$ be a vector space over $\C$ together with a symmetric non-degenerate bilinear form $(-,-)$. Even if in this paper we will only be concerned with the orthogonal group $O(V)$, we will denote the isometry group with respect to $(-,-)$ as $G(V) := O(V)$ in order to keep the notation as close as possible with \cites{kraft1982geometry, ohta1986singularities}. Let $\gl(V)$ be the space of linear endomorphisms of $V$. Clearly $G(V)$ acts on $\gl(V)$ by conjugation. We denote by $D^*$ the adjoint of an endomorphism $D$ with respect to $(-,-)$. The endomorphism $\theta: D \mapsto -D^*$ of $\gl(V)$ is involutive, therefore we have a decomposition into eigenspaces for $\theta$ as follows:
\[
    \gl(V) = \kl(V) \oplus \p(V),
\]
where $\kl(V) = \{D: D = \theta(D) = -D^*\}$ is the Lie algebra of $G(V)$. Moreover the action of $G(V)$ leaves both $\kl(V)$, $\p(V)$ stable.

In this paper we are concerned with the study of nilpotent orbits of $G(V)$ in $\p(V)$. It is well known (\cite{sekiguchi1984nilpotent}*{Sec. 3.1}, \cite{ohta1986singularities}*{Sec. 0}) that the nilpotent orbits in $\p(V)$ are completely determined by their Jordan form and every partition of $n:=\dim V$ realizes a non-empty nilpotent orbit. We denote by $P(n)$ the set of partitions of $n$ and by $C_\lambda \subseteq \p(V)$ the nilpotent orbit corresponding to $\lambda$.

Let $\lambda$ be a partition. We set $|\lambda| := n$ if $\lambda \in P(n)$ and we denote the dual partition of $\lambda$ by $\widehat \lambda$. We frequently identify a partition with its Young diagram. This means that, if $\lambda = (\lambda_1, \dots, \lambda_h)$, $|\lambda| = \lambda_1+\dots+\lambda_h$, $\lambda_i$ are the rows of $\lambda$ and if $\widehat \lambda = (\widehat\lambda_1, \dots, \widehat\lambda_t)$, $\widehat\lambda_j$ are the columns of $\lambda$.

We recall the dimension formula for the orbit $C_\lambda$ from \cite{ohta1986singularities}*{Remark 8}:
\begin{equation}
    \label{equ:sno:dim}
    \dim C_\lambda = \frac 1 2 \left(n^2 - \sum_{i=1}^t {\widehat\lambda}_i^{\,2}\right).
\end{equation}

As the nilpotent cone of $\p(V)$ is $G(V)$-stable with only finitely many orbits, we have that orbit closure $\overline{C_\lambda}$ is $G(V)$-stable, and the complement $\overline{C_\lambda} \setminus C_\lambda$ is a disjoint union of finitely many orbits. The relation $C_\mu \subseteq \overline{C_\lambda}$ produces a partial order on the partitions, called \textit{dominance order} and denoted by $\mu \leq \lambda$, and it is given by
\[
    \sum_{i=1}^j \lambda_i \geq \sum_{i=1}^j \mu_i\quad \forall j \in \{1, \dots, h\}
\]
or, equivalently,
\[
    \sum_{k>j} \widehat\lambda_k \geq \sum_{k>j} \widehat\mu_k \quad \forall j \in \{1, \dots, t\}.
\]

We recall the notion of nilpotent pairs. We follow the same argument of \cite{kraft1982geometry}. Let $V$ (resp. $U$) be a finite dimensional vector space over $\C$ equipped with a non-degenerate symmetric bilinear form $( - , - )_V$ (resp. $( - , - )_U$). We denote $L(V,U) := \Hom_\C(V,U)$, $L(V) := L(V,V)$ and we define $L_{V,U} := L(V,U) \times L(U,V)$. We can interpret $L_{V,U}$ as the representation variety of the quiver
\[
    Q = \begin{tikzcd}
1 \arrow[r,bend left] & \arrow[l,bend left] 2
\end{tikzcd}
\]
with dimension vector $\underline{n}=(\dim V, \dim U)$. This means that
\[
L_{V,U} = \left\{
\begin{tikzcd}
V \arrow[r,bend left, "A"] & \arrow[l,bend left, "B"] U
\end{tikzcd}
: A \in L(V,U), B\in L(U,V)
\right\}.
\]
The group $G(V)\times G(U)$ acts  on $L_{V,U}$ by change of basis.

For every $A \in L(V,U)$ we define the adjoint map $A^* \in L(U,V)$ as the unique map such that
\begin{equation}
    \label{equ:sno:adj}
    ( Av , u )_U = ( v , A^*u )_V
\end{equation}
for all $v \in V$, $u \in U$.

\begin{lemma}
    \label{lem:sno:ranks}
    Let $A \in L(V,U)$ be a linear map. Then
    \[
        \rk A = \rk A^*.
    \]
\end{lemma}
\begin{proof}
    By (\ref{equ:sno:adj}), we get that $\ker A \bot \operatorname{Im} A^*$ and $\ker A^* \bot \operatorname{Im} A$.
\Qed\end{proof}

A pair $(A,B) \in L_{V,U}$ is nilpotent if the endomorphism $AB$ (or equivalently $BA$) is nilpotent. A pair $(A,B) \in L_{V,U}$ is symmetric if $B = A^*$. We define $N_{V,U}$ as the cone in $L_{V,U}$ of the symmetric nilpotent pairs.

As in \cite{kraft1982geometry}, we have maps
\[
\begin{tikzcd}
L_{V,U} \arrow[r,"\pi"]\arrow[d,swap,"\rho"] & L(U)\\
L(V)
\end{tikzcd}
\]
defined by $\pi(A,B) = AB$, $\rho(A,B) = BA$. We can restrict $\pi,\rho$ to the subspace of symmetric pairs, i. e. the image of $L(V,U)\rightarrow L_{V,U}$, $A \mapsto (A,A^*)$:
\[
\begin{tikzcd}
L(V,U) \arrow[r,"\pi"]\arrow[d,swap,"\rho"] & \p(U)\\
\p(V).
\end{tikzcd}
\]

The subvariety $N_{V,U} \subseteq L_{V,U}$ is $G(V)\times G(U)$-stable. We call \emph{nilpotent symmetric orbit} each of the $G(V)\times G(U)$-orbits in $N_{V,U}$. We recall \cite{kraft1979closures}*{Sec. 4.2} that an $ab$-diagram is a list of $ab$-strings i.e. strings with letters $a$ and $b$ occurring on alternate positions. As shown in \cite{kraft1979closures}*{Sec. 4.3}, to each nilpotent pair $(A, B) \in L_{V,U}$ corresponds an $ab$-diagram which determines its $GL(V)\times GL(U)$-orbit completely; i. e. the $GL(V)\times GL(U)$-orbits correspond one to one to the $ab$-diagrams having $\dim V$ $a$'s and $\dim U$ $b$'s. The following picture illustrates how to associate an ab-diagram $\delta=\delta_{(A,B)}$ to a nilpotent pair $(A,B)$:
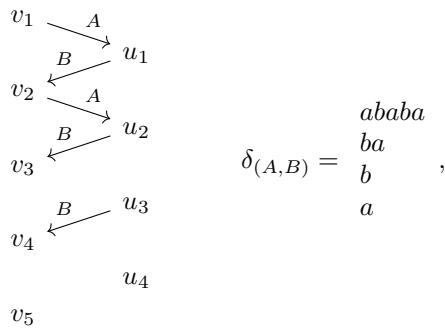
\[
    \begin{array}{cc}
        \begin{tikzcd}[row sep=1pt]
            v_1 \arrow[dr, "A"]\\
            & u_1 \arrow[dl, "B"']\\
            v_2 \arrow[dr, "A"]\\
            & u_2 \arrow[dl, "B"']\\
            v_3\\
            & u_3 \arrow[dl, "B"']\\
            v_4\\
            & u_4\\
            v_5
        \end{tikzcd}
        &
        \qquad\delta_{(A,B)} =
            \begin{array}{l}
                ababa\\
                ba\\
                b\\
                a
            \end{array},
    \end{array}
\]
where $\{v_1,\cdots, v_5\}$ is a suitable basis of $V$ and $\{u_1,\cdots, u_4\}$ is a suitable basis of $U$.

If $\delta$ is the $ab$-diagram of an element $p \in N_{V,U}$, we can retrieve the Young diagram of $\pi(p)$ (resp. $\rho(p)$) by suppressing the $a$'s (resp. the $b$'s) from $\delta$. We denote by $\pi(\delta)$ (resp. $\rho(\delta)$) the Young diagram obtained in this manner.

Let $(A,B)\in L_{V,U}$ be a nilpotent pair with $ab$-diagram $\delta$. For every $h\geq 1$, we want to infer the rank of the linear map $(AB)^{h-1}A:V\rightarrow V$ (resp. $(BA)^{h-1}B:U\rightarrow U$) from $\delta$. If $\gamma = c_1 \cdots c_\ell$ is a string, for any $1\leq i\leq j \leq \ell$, we call $c_i\cdots c_j$ a \emph{substring} of $\gamma$. A \emph{substring} of $\delta$ is a substring of one of its $ab$-strings. For example, for $h=1$, the rank of $A$ equals the number of occurrences of the letter $a$ which are not in the last position of a string; this is hence the number of occurrences of the sub-string $(ab)=(ab)^{h}$ in $\delta$. With this in mind we prove the following lemma.

\begin{lemma}\label{lem:sno:rankaba}
    Let $(A, B)$ be a nilpotent pair with $ab$-diagram $\delta$. For every $h\geq 1$, the $\rk((AB)^{h-1}A)$ (resp. $\rk((BA)^{h-1}B)$) equals the number of occurrences of the substring $(ab)^h$ (resp. $(ba)^h$) in $\delta$.
\end{lemma}
\begin{proof}
    The $ab$-diagram $\delta$ suggests the choice of basis of $V$ and $U$ for which the behaviour of $A$ and $B$ is straightforward. For each $ab$-string $\delta_i=c_{i1}c_{i2}\dots c_{i\ell_i}$ of $\delta$ (where $c_{ij} \in \{a, b\}$), for each $j=1,\dots,\ell_i$, we pick vectors $v_{ij} \in V$ (resp. $u_{ij} \in U$) if $c_{ij} = a$ (resp. $c_{ij} = b$), with the following properties: $\{v_{ij}\}_{i,j}$ (resp. $\{u_{ij}\}_{i,j}$) is a basis of $V$ (resp. $U$) and $Av_{ij} = u_{i,{j+1}}$ (resp. $Bu_{ij} = v_{i,j+1}$) if $j+1 \leq \ell_i$ or $Av_{i\ell_i}=0$ (resp. $Bu_{i\ell_i}=0$).
    
    Notice that the map $(AB)^{h-1}A$ carries $v_{ij}$ to $u_{i(j+2h-1)}$ if $j+2h-1\leq \ell_i$ or to $0$ otherwise. Therefore, $\rk((AB)^{h-1}A)$ equals the number of pairs $(i,j)$ such that $c_{ij} = a$ and $j+2h-1\leq \ell_i$. Moreover, every such pair $(i,j)$ corresponds one to one to the substring $c_{ij}\dots c_{i(j+2h-1)}$ of $\delta$, which is the substring $(ab)^h$. A similar argument computes $\rk((BA)^{h-1}B)$.
    \Qed
\end{proof}

\section{Ortho-symmetric \texorpdfstring{$ab$}{ab}-diagrams}
\label{sec:ortoab}
In this section we give a combinatorial description of the $ab$-diagrams of the elements of $N_{V,U}$.

Let $X \in N_{V,U}$ be a nilpotent symmetric pair, let $G = G(V) \times G(U)$ and let $G.X$ be its orbit in $N_{V,U}$. We also consider $G' = \GL(V) \times \GL(U)$ acting on $L_{V,U}$, and the orbit $G'.X \subseteq L_{V,U}$.

The automorphism $\sigma$ of $G'$ defined by
\[
    (g_1, g_2)^\sigma = ((g_1^*)^{-1},(g_2^*)^{-1})
\]
has $G$ as the set of fixed points. The automorphism (still denoted by $\sigma$) of $L_{V,U}$ defined by
\[
    \sigma(A,B) = (B^*, A^*)
\]
has the symmetric pairs as the set of fixed points. It is straightforward to check that all hypothesis of \cite{magyar2000symplectic}*{Prop. 2.1} hold. Thus we get that:
\[
    G'.X \cap N_{V,U} = G.X.
\]

Let $\delta$ be an $ab$-diagram associated to an orbit $G'.Y$. We say that $\delta$ is \textit{ortho-symmetric} if
\[
    G'.Y \cap N_{V,U} \neq \emptyset.
\]

We are left with the following problem: which $ab$-diagrams are \textit{ortho-symmetric}?

For any two pairs $(A, B) \in L_{V,U}$, $(A', B') \in L_{V', U'}$, the direct sum $(A, B)\oplus (A', B') = (A \oplus A', B \oplus B')$ is defined as an element in $L_{V\oplus V', U\oplus U'}$. If $\delta$ is the $ab$-diagram of $(A,B)$ and $\delta'$ is the $ab$-diagram of $(A',B')$, the $ab$-diagram of the direct sum $(A,B)\oplus(A',B')$ is the disjoint union of $\delta$ and $\delta'$.

We call a symmetric pair $(A, A^*)$ indecomposable if it can not be written as a direct sum of two nontrivial symmetric pairs. We claim that the $ab$-diagrams of indecomposable nilpotent symmetric pairs are given in \cref{tab:ortoab:indec}. This table closely follows table II in \cite{kraft1982geometry}*{6.3}.

\begin{table}[hbt]
\begin{center}
    \begin{tabular}{l c c c}
        type & $\alpha_k$ & $\beta_k$ & $\varepsilon_k$ \\
        \hline
        & $(ab)^ka = abab\cdots a$ & $(ba)^kb = baba\cdots b$ &
        $\begin{array}{l}
            (ab)^k = abab\cdots b \\
            (ba)^k = baba\cdots a \\
            \end{array}$ \\
        \hline
        \#$a$ & $k+1$ & $k$ & $2k$ \\
        \#$b$ & $k$ & $k+1$ & $2k$ \\
        \hline
        & $k \geq 0$ & $k \geq 0$ & $k \geq 1$
    \end{tabular}
\end{center}
\caption{
    \label{tab:ortoab:indec}
    Indecomposable ortho-symmetric $ab$-diagrams.
}
\end{table}

\begin{proposition}
    An $ab$-diagram $\delta$ is ortho-symmetric if and only if it is a disjoint union of finitely many $ab$-diagrams from  \cref{tab:ortoab:indec}.
\end{proposition}

\begin{proof}
    We say that an $ab$-diagram $\eta$ has \emph{property $\mathcal P$} if, for every $h \geq 1$, the number of occurrences of $(ab)^h$ in $\eta$ equals the number of occurrences of $(ba)^h$ in $\eta$.

    Every $ab$-diagrams $\alpha_k, \beta_k, \varepsilon_k$ in \cref{tab:ortoab:indec} has property $\mathcal P$. Indeed in the diagrams $\alpha_k$ and $\beta_k$, for every $h \geq 1$, there are exactly $k-h+1$ occurrences of $(ab)^h$ and of $(ba)^h$. The diagram $\varepsilon_k$ has property $\mathcal P$ by symmetry of its two $ab$-strings.

    Let $(A, A^*) \in N_{V,U}$ be a nilpotent symmetric pair with $ab$-diagram $\delta$, an ortho-symmetric $ab$-diagram. For every positive integer $h$, we notice that the pair $((AA^*)^{h-1}A, (A^*A)^{h-1}A^*)$ is nilpotent symmetric. By \cref{lem:sno:ranks}, we have:
    \[
        \rk (AA^*)^{h-1} A = \rk (A^*A)^{h-1} A^*,
    \]
    which, by \cref{lem:sno:rankaba}, implies that $\delta$ has property $\mathcal P$.    
    
    We claim that an $ab$-diagram $\delta$ with property $\mathcal P$ is a disjoint union of finitely many $ab$-diagrams from \cref{tab:ortoab:indec}. We proceed by contradiction: suppose that there exist $ab$-diagrams with property $\mathcal P$ which are not disjoint union of finitely many $ab$-diagrams from \cref{tab:ortoab:indec} and take $\delta$ one with the least number of $ab$-strings. Let $s$ be one of the $ab$-strings in $\delta$ with maximal length and let $l$ be its length.
    
    If $l = 2k+1$ is odd, $s = \alpha_k$ or $s = \beta_k$. Let $\delta'$ be the $ab$-diagram such that $\delta = s \sqcup \delta'$ (we cannot have $\delta = s$, otherwise $\delta$ is already in \cref{tab:ortoab:indec}).
    
    If $l = 2k$ is even, without loss of generality, we assume that $s$ starts with $a$. Then, in $s$ there exists exactly one occurrence of $(ab)^k$ and no occurrence of $(ba)^k$. As $\delta$ satisfies property $\mathcal P$, by maximality of $l$, there must be an $ab$-string $s'$ in $\delta$ of length $l=2k$ starting with $b$. Thus $s \sqcup s' = \varepsilon_k$, so let $\delta'$ be the $ab$-diagram such that $\delta = s \sqcup s' \sqcup \delta'$ (again, it cannot be empty).
    
    In both cases, $\delta$ and $\delta \setminus \delta'$ have property $\mathcal P$, therefore also $\delta'$ has property $\mathcal P$. By minimality of $\delta$, $\delta'$ is disjoint union of finitely many $ab$-diagrams from \cref{tab:ortoab:indec}, thus so is $\delta$. This contradicts the existence of such a $\delta$ and concludes the only if part of the proposition.

    On the other hand, we show that every $ab$-diagram in \cref{tab:ortoab:indec} is ortho-symmetric. In order to show this, we just need to construct a symmetric pair $(A, A^*)$ associated to each diagram $\alpha_k$, $\beta_k$, $\varepsilon_k$.

    \begin{description}
        \item[$\alpha_k$:] Let $D: \C^{k+1} \rightarrow \C^{k+1}$ be a symmetric nilpotent endomorphism of nilpotent order exactly $k+1$. Let $D = I \circ X$ be the canonical decomposition of the map $D$ through its image $D(\C^{k+1})$, so that $X: \C^{k+1} \rightarrow D(\C^{k+1})$ and $I: D(\C^{k+1}) \hookrightarrow \C^{k+1}$. Then $D$ induces a non-degenerate bilinear symmetric form in $D(\C^{k+1})$, as shown in \cite{kraft1982geometry}*{4}, and $X = I^*$. As $\dim D(\C^{k+1}) = k$ and both $X, I$ have maximal rank, we immediately get that the $ab$-diagram of the symmetric nilpotent pair $(X, I)$ is $\alpha_k$.
        \item[$\beta_k$:] Let $(A, A^*)$ be a pair with $ab$-diagram of type $\alpha_k$. Then $(A^*, A)$ is a pair with $ab$-diagram of type $\beta_k$.
        \item[$\varepsilon_k$:] Let $(A, B)$ be a nilpotent (not symmetric) pair between the spaces $(V, U) = (\C^k, \C^k)$ with $ab$-diagram $ab\cdots b$. Let $(B^*, A^*)$ be the linear dual of $(A, B)$, so that it is a nilpotent pair between the dual spaces $(V^*, U^*)$ with $ab$-diagram $ba\cdots a$. Then we can equip $V \oplus V^*$ and $U \oplus U^*$ with the non-degenerate symmetric bilinear form given by the duality pairing. Therefore $(A, B) \oplus (B^*, A^*)$ is a symmetric nilpotent pair and its $ab$-diagram has type $\varepsilon_k$.\Qed
    \end{description}
\end{proof}

\section{The variety \texorpdfstring{$Z$}{Z}}
\label{sec:zvar}

For each partition $\lambda$ of $n$ we are going to construct a variety $Z^{(\lambda)}$ (or just $Z$, if the partition $\lambda$ is clear from the context) with the following property: there exists a quotient map $Z \rightarrow \overline{C_\lambda}$. In this way, we will be able to assert some properties of $\overline{C_\lambda}$ by proving them for the variety $Z$. The construction of $Z$ is analogous to the one carried out in \cite{kraft1982geometry} and it has already been described in \cite{ohta1986singularities}; here we just recall its definition.

Let $\lambda \in P(n)$ be a partition and let $t = \lambda_1$ be the number of columns of $\lambda$. As in \cite{kraft1982geometry}, we define a variety $Z$ using the representations of the quiver:
\[
Q_t =
\begin{tikzcd}
0 \arrow[r,bend left,"x"] & \arrow[l,bend left,"y"] 1 \arrow[r,bend left,"x"] & \arrow[l,bend left,"y"] \cdot\cdot \arrow[r,bend left,"x"] & \arrow[l,bend left,"y"] t.
\end{tikzcd}
\]
Let
\[
    n_i = \sum_{j>i}^t \widehat\lambda_j
\]
so that $n_0 = n$ and $n_t=0$. We fix the dimension vector $\underline{n} = (n_0, \dots, n_t)$ and vector spaces $V_i$ such that $\dim V_i = n_i$. We also fix a symmetric bilinear non-degenerate form on each $V_i$.

The variety $Z$ consists of the representations of dimension vector $\underline{n}$ of the quiver $Q_t$ with relations:
\begin{itemize}
    \item $xy = yx$;
    \item $y = x^*$ (the adjoint map as in (\ref{equ:sno:adj})).
\end{itemize}
Therefore, each point $z\in Z$ is a sequence of maps
\[
    (A_1, B_1; A_2, B_2; \dots; A_t, B_t)
\]
where $A_i \in L(V_{i-1}, V_i)$, $B_i \in L(V_i, V_{i-1})$, $A_iB_i = B_{i+1}A_{i+1}$ and $B_i = A_i^*$ for each $i$. As each $B_i$ is adjoint to $A_i$, we get an inclusion
\[
    Z \hookrightarrow L(V_0, V_1) \times \dots \times L(V_{t-1}, V_t)
\]
mapping
\[
    (A_1, B_1; A_2, B_2; \dots; A_t, B_t) \mapsto (A_1, \dots, A_t).
\]

As in \cite{kraft1982geometry}, we can think of $Z$ as a schematic fiber in the following way. Let
\begin{align*}
    M &:= L(V_0, V_1) \times \dots \times L(V_{t-1},V_t),\\
    N &:= \p(V_1) \times \dots \times \p(V_{t-1}),\\
    \Phi &: M \rightarrow N\\
    & (A_1, \dots, A_t) \mapsto (A_1A_1^*-A_2^*A_2, \dots, A_{t-1}A_{t-1}^* - A_t^*A_t);
\end{align*}
then we put $Z := \Phi^{-1}(0)$. As $M$ is an affine space, this interpretation will be useful in \cref{sec:ci} to show that $Z$ is a complete intersection variety (at least for some $\lambda$).

As in \cite{kraft1982geometry}, let $G = G(V_0) \times \dots \times G(V_t)$ be the group which acts on $Z$ by change of basis and $H = G(V_1) \times \dots \times G(V_t) \subseteq G$. We have a map
\begin{align*}
    \Theta:& Z \rightarrow \overline{C_\lambda}\\
    &(A_1, B_1; \dots; A_t, B_t) \mapsto B_1A_1.
\end{align*}
By the same argument as in \cite{kraft1982geometry}, we have the following:
\begin{proposition}
    The map $\Theta$ splits through the quotient by the group $H$ as
    \[
        \Theta: Z \longrightarrow Z/H \tilde\longrightarrow \overline{C_\lambda};
    \]
    moreover $\Theta$ is $G(V_0) = G/H$-equivariant.
\end{proposition}
(In particular, $\Theta^{-1}(C_\mu)$ is stable under the action of $G(V_0)$ for each partition $\mu \leq \lambda$.)

By restricting the quiver $Q_t$ to two consecutive vertices $i-1, i$, we get a map $Z \rightarrow N_{V_{i-1},V_i}$. As all the orbits in $N_{V_{i-1},V_i}$ are defined by their ortho-symmetric $ab$-diagram, we can map each point $z \in Z$ and $i = 1,\dots,t$ to its $i$-th $ab$-diagram $\tau_i$. Therefore, for each string of $ab$-diagrams
\[
    \tau = (\tau_1, \dots, \tau_t)
\]
we define
\[
    Z_\tau = \{z \in Z:\forall i,\text{ the $i$-th $ab$-diagram of $z$ is $\tau_i$ }\}.
\]
The following are necessary conditions on $\tau$ in order to have $Z_\tau \neq \emptyset$:
\begin{enumerate}
    \item $\tau_i$ must have exactly $n_{i-1}$ $a$'s and $n_i$ $b$'s,
    \item $\rho(\tau_{i}) = \pi(\tau_{i+1})$ for $i=1,\dots,t-1$,
    \item each $\tau_i$ must be ortho-symmetric.
\end{enumerate}
We will denote by $\Lambda^{(\lambda)}$ (or, shortly, by $\Lambda$) the set of all strings $\tau$ which verify the three conditions above.

As in \cite{kraft1979closures}*{Lemma 5.4}, \cite{kraft1982geometry}*{8.2}, $\Theta^{-1}(C_\lambda)$ is given by exactly one stratum $Z_{\tau^0}$. The string of $ab$-diagrams $\tau^0$ can be constructed in the following way: for all $i=0,\dots,t-1$, let $\lambda_i$ be the partition of $n_i$ given by the columns of $\lambda$ with indices greater than $i$; then let $\tau_{i+1}$ be the $ab$-diagram built by setting one $a$ on each box in the Young diagram of $\lambda_i$ and by setting one $b$ between each pair of consecutive $a$'s.

\begin{example}
    Let $\lambda = (3,1)$ be a partition. Then the string of $ab$-diagrams $\tau^0$ is the following:
    \[
    \tau^0 = (\tau^0_1, \tau^0_2, \tau^0_3) = \left(
        \begin{array}{l|l|l}
            ababa & aba & a \\
            a & &
        \end{array}
        \right).
    \]
\end{example}

The stratum $Z_{\tau^0}$ is open in $Z$ as it is the subvariety of the points
\[
    (A_1, B_1, \dots, A_t, B_t)
\]
in $Z$ where each $A_i$, $B_i$ has maximal rank. We have the following key result after the manner of Kraft, Procesi and Ohta.
\begin{proposition}[\cite{kraft1982geometry}*{5.5} and \cite{ohta1986singularities}*{Prop. 8}]
    \label{prop:zvar:zsmooth}
    The open set $Z_{\tau^0}$ is contained in the smooth locus of $Z$.
\end{proposition}

We recall a dimension formula for the strata $Z_\tau$ due to Otha \cite{ohta1986singularities}*{Prop. 6}. First, for each $ab$-diagram $\delta$, we denote by $a_i$ (resp. $b_i$) the number of rows of $\delta$ starting by $a$ (resp. $b$) of length $i$. We define
\[
    \Delta(\delta) := \sum_{i \text{ odd }}a_ib_i
\]
and
\[
    o(\delta) := \sum_{i \text{ odd }}(a_i + b_i).
\]
\begin{proposition}{\cite{ohta1986singularities}*{Prop. 6}}
    Let $\tau = (\tau_1, \dots, \tau_t)$. We have
    \begin{equation}
        \label{equ:zvar:dimz}
        \begin{split}
            \dim Z_\tau =
            \frac 1 2 \dim C_{\pi(\tau_1)} 
            + \sum_{i=0}^{t-1}\left(\frac 1 2 n_in_{i+1} - \frac 1 4 (n_i + n_{i+1})\right) \\
            + \sum_{i=1}^t \left(\frac 1 4  o(\tau_i) -\frac 1 2\Delta(\tau_i)\right).
        \end{split}
    \end{equation}
\end{proposition}

In the following sections we focus on the combinatorics of the $\tau \in \Lambda$. We need tools to effectively manipulate formula (\ref{equ:zvar:dimz}), in particular we aim at \cref{prop:comb:big} and \cref{prop:comb:bigr} (see \cref{sec:comb}).

\section{\texorpdfstring{$s$}{s}-step condition}
\label{sec:sstep}
In this section we define a condition which plays a pivotal role in the subsequent study of the nilpotent symmetric orbits.

\begin{definition}
    Given a partition $\lambda = (\lambda_1, \dots, \lambda_h)$, we say that $\lambda$ satisfies the \emph{$s$-step condition} if for every $i = 1, \dots, h$ we have
    \[
        \lambda_i \leq \lambda_{i+1} + s
    \]
    with the convention that $\lambda_{h+1} = 0$.
\end{definition}

We want to highlight some simple properties of this definition. First of all, $s_1$-step condition implies $s_2$-step condition for each $s_2 \geq s_1$. Secondly, if $\lambda$ satisfies the $s$-step condition, then if we remove the first row or the first column from $\lambda$, the resulting partition still satisfies the $s$-step condition.

\begin{example}
    The single line partition $(n)$ satisfies the $n$-step condition, but not the $(n-1)$-step condition. The triangular partition $(n, n-1, \dots, 1)$ satisfies the $1$-step condition. The \textit{2-skew} triangular partition $(2n, 2(n-1), \dots, 2)$ satisfies the $2$-step condition (but not the $1$-step condition).
\end{example}

\section{Differences between partitions}
\label{sec:diff}

In this section we study some quantitative ways to address the difference between two partitions. Let $\lambda, \mu \in P(n)$ be two partitions of $n$ such that $\lambda > \mu$.

\begin{definition}
    If $\lambda = (\lambda_1, \dots, \lambda_k) \in P(n)$ and $\mu = (\mu_1, \dots, \mu_k) \in P(n)$, possibly allowing $\lambda_i = 0$ for some $i$, we define
    \[
        q(\lambda, \mu) = \frac 1 2 \sum_{i=1}^k |\lambda_i - \mu_i|.
    \]
\end{definition}
\begin{remark}
    By an easy computation modulo 2, it can be seen that $q(\lambda, \mu)$ is an integer.
\end{remark}

\begin{example}
    Let $\lambda = (7, 2, 2, 1)$ and $\mu = (5, 3, 1, 1, 1, 1)$. We can draw the Young diagram of both partitions and overlay one over the other, so that $\lambda$ is given by white and gray boxes while $\mu$ is given by white and black boxes in the following diagram:
    \[
        \ytableaushort{\none}
        * {7,3,2,1,1,1}
        * [*(gray)]{5+2,0,1+1}
        * [*(black)]{0,2+1,0,0,1,1}
    \]

    We can compute the $q$ difference:
    \[
        q(\lambda, \mu) =\frac 1 2 \left( (7-5) + (3-2) + (2-1) + (1-1) + 1 + 1 \right) = 3.
    \]
    We can check that in this case $q(\lambda, \mu)$ is the number of gray boxes or the number of black boxes. This is an easy fact (it is not even needed that $\lambda > \mu$) that can be derived from the following lemma.
\end{example}

\begin{lemma}
    \label{lem:diff:ind}
    Let $\lambda > \mu$ be two partitions of $n$. Then there exist partitions
    \[
    \lambda = \lambda^0 > \lambda^1 > \dots > \lambda^{q(\lambda, \mu)} = \mu
    \]
    such that $q(\lambda^i, \lambda^{i+1}) = 1$ for each $i = 0, \dots, q(\lambda, \mu)-1$.
    
    Moreover, we can choose $\lambda^i$ such that for each row $r$ and for each column $c$ the sequences of integers $\lambda^0_r, \dots, \lambda^{q(\lambda, \mu)}_r$ and $\widehat \lambda^0_c, \dots, \widehat\lambda^{q(\lambda, \mu)}_c$ are monotone.
\end{lemma}
\begin{proof}
    We proceed by induction on $q=q(\lambda, \mu)$. If $q = 1$ we trivially have $\lambda^1 = \mu$. If $q > 1$, we are going to build $\lambda^1$ such that $q(\lambda, \lambda^1) = 1$ and $q(\lambda^1, \mu) = q(\lambda, \mu) - 1$.
    
    We build $\lambda^1$ starting from $\lambda$ in this way. Let $\tilde \imath$ be the first index such that $\lambda_{\tilde \imath} > \mu_{\tilde \imath}$ and let $j$ be the first index such that $\lambda_j < \mu_j$. Clearly there must exists such indices because $\lambda \neq \mu$ and $|\lambda| = |\mu|$. Let $i \geq \tilde \imath$ be the last index such that $\lambda_i = \lambda_{\tilde \imath}$ (possibly $i = \tilde \imath$). As $\lambda > \mu$, we must have $\tilde \imath < j$ and therefore $i < j$. We define
    \[
        \lambda^1 = (\lambda_1, \dots, \lambda_{i-1}, \lambda_i-1, \lambda_{i+1}, \dots, \lambda_{j-1}, \lambda_j + 1, \lambda_{j+1}, \dots, \lambda_k).
    \]
    The sequence $\lambda^1$ is actually a partition of $n$ as $|\lambda^1| = |\lambda| = n$, $\lambda_i-1 \geq \lambda_{i+1}$ by the maximality of $i$ and $\lambda_{j-1} \geq \lambda_j+1$ by the minimality of $j$.
    
    Moreover, both the following hold: $q(\lambda, \lambda^1) = 1$ and $q(\lambda^1, \mu) = q-1$. We apply induction on the pair $(\lambda^1, \mu)$ and we conclude the proof of the first part of the lemma.

    To conclude the proof, we notice that if, for a row $r$, $\lambda_r > \mu_r$ (resp. $\lambda_r < \mu_r$), then we have $\lambda^i_r \geq \lambda^{i+1}_r$ (resp. $\lambda^i_r \leq \lambda^{i+1}_r$) by construction. The same holds true for the columns lengths.
\Qed\end{proof}

\begin{example}
    Taking the partitions $\lambda$, $\mu$ defined in the previous example, the construction explained in the lemma produces the following sequence of partitions:
    \[
    (7, 2, 2, 1) > (6, 3, 2, 1) > (5, 3, 2, 1, 1) > (5, 3, 1, 1, 1).
    \]
\end{example}

\begin{remark}
    In view of \cref{lem:diff:ind} we can interpret $q(\lambda, \mu)$ as the minimum number of boxes needed to lower in order to obtain $\mu$ from $\lambda$.
\end{remark}

We consider two other ways of comparing partitions. Given a partition $\lambda$, for each box $B$ in the Young diagram of $\lambda$ we define its column number $c(B)$ (resp. row number $r(B)$) by counting the columns (resp. rows) starting from the leftmost column (resp. uppermost row). For example, given $\lambda = (7, 2, 2, 1)$ we have:

\begin{center}
    $c(-)$: \ytableaushort{1 2 3 4 5 6 7, 1 2, 1 2, 1} \qquad
    $r(-)$: \ytableaushort{1 1 1 1 1 1 1, 2 2, 3 3, 4}
\end{center}

\begin{remark}
    By summing up over the rows (resp. the columns) of $\lambda$, we have:
    \begin{equation}
        \label{equ:diff:trivsum}
        \sum_{B \in \lambda} c(B) = \sum_{i=1}^h \binom{\lambda_i+1}{2} \quad \left(\text{resp. } \sum_{B \in \lambda} r(B) = \sum_{i=1}^t\binom{\widehat{\lambda}_i+1}{2}\right).
    \end{equation}
\end{remark}

\begin{definition}
    Given $\lambda \geq \mu$ as before, we define:
    \[
        c(\lambda, \mu) = \sum_{B \in \lambda} c(B) - \sum_{B \in \mu} c(B),
    \]
    where $B$ is selected among the boxes in the Young diagrams of $\lambda$ and $\mu$ and, similarly,
    \[
        r(\lambda, \mu) = \sum_{B \in \mu} r(B) - \sum_{B \in \lambda} r(B).
    \]
\end{definition}

\begin{example}
Let $\lambda = (7, 2, 2, 1)$ and $\mu = (5, 3, 1, 1, 1, 1)$. We compute
\begin{align*}
    c(\lambda, \mu) & = (28 + 3 + 3 + 1) - (15 + 6 + 1 + 1 + 1 + 1) = 10, \\
    r(\lambda, \mu) & = (21 + 3 + 3 + 1 + 1) - (10 + 6 + 1 + 1 + 1 + 1 + 1) = 8.
\end{align*}
\end{example}

\begin{remark}
    It is always guaranteed that $c(\lambda, \mu) \geq 0$ and $r(\lambda, \mu) \geq 0$ for every $\lambda \geq \mu$. We even have that
    \begin{equation}
        \label{equ:diff:equcases}
        \lambda = \mu \Longleftrightarrow c(\lambda, \mu) = 0 \Longleftrightarrow r(\lambda, \mu) = 0.
    \end{equation}
    
    In fact we have that $c(\lambda, \mu) \geq 0$ is equivalent to
    \begin{equation}
        \label{equ:diff:karam}
        \sum_i \binom{\lambda_i+1}2 \geq \sum_i \binom{\mu_i+1}2.
    \end{equation}
    As the function $n \mapsto \binom{n+1}{2}$ is convex, (\ref{equ:diff:karam}) is given by the Karamata's inequality (\cite{beckenbach2012inequalities}*{Chapter 1, \S 28}).
\end{remark}

\begin{remark}
    Given partitions $\lambda \geq \mu \geq \nu$, by definition of $c$ and $r$ we have:
    \begin{align}
        \label{equ:diff:cadd}
        c(\lambda, \nu) &= c(\lambda, \mu) + c(\mu, \nu), \\
        \label{equ:diff:radd}
        r(\lambda, \nu) &= r(\lambda, \mu) + r(\mu, \nu).
    \end{align}
\end{remark}

\begin{remark}
    For every pair of partitions $\lambda \geq \mu$ we have that
    \begin{equation}
        \label{equ:diff:ineq}
        c(\lambda, \mu) \geq q(\lambda, \mu).
    \end{equation}
    In fact, when $q(\lambda, \mu) = 1$, by (\ref{equ:diff:equcases}) we get (\ref{equ:diff:ineq}). Otherwise, by \cref{lem:diff:ind}, we have a sequence of partitions $\lambda^0, \dots, \lambda^{q(\lambda, \mu)}$. For each pair $(\lambda^i, \lambda^{i+1})$, we have $c(\lambda^i, \lambda^{i+1}) \geq 1$, therefore the conclusion follows by summing up each inequalities by (\ref{equ:diff:cadd}).
\end{remark}

\begin{lemma}
    \label{lem:diff:usef}
    Let $\lambda$ be a partition satisfying the $s$-step condition and $\mu < \lambda$ be another partition. Then
    \begin{equation}
        \label{equ:diff:sstep}
        s \cdot r(\lambda, \mu) \geq c(\lambda, \mu) + q(\lambda, \mu).
    \end{equation}
    Moreover, (\ref{equ:diff:sstep}) holds strictly if there exists a column index $c$ such that $\widehat \mu_c > \widehat \lambda_c + 1$ or a row index $r$ such that $\mu_r >  \lambda_r + 1$.
\end{lemma}
\begin{proof}
    Let $q = q(\lambda, \mu)$. By \cref{lem:diff:ind}, we obtain a sequence
    \[
        \lambda = \lambda^0 > \dots > \lambda^q = \mu,
    \]
    with $q(\lambda^i, \lambda^{i+1})=1$ for $i = 0, \dots, q-1$.
    
    We want to prove
    \begin{equation}
        \label{equ:diff:onestep}
        s \cdot r(\lambda^i, \lambda^{i+1}) \geq c(\lambda^i, \lambda^{i+1}) + 1
    \end{equation}
    for all $i = 0, \dots, q - 1$.
    Once (\ref{equ:diff:onestep}) is proved, the conclusion will follow by summing up each of those inequalities by (\ref{equ:diff:radd}) and (\ref{equ:diff:cadd}).

    We start by noticing that the $s$-step condition implies that, for every pair of row indices $r_1 < r_2$,
    \[
        \lambda_{r_1} - \lambda_{r_2} \leq s (r_2 - r_1).
    \]
    Fix $r_1$ to be an index row such that $\lambda_{r_1} \geq \mu_{r_1}$ and fix $r_2$ to be an index row such that $\lambda_{r_2} \leq \mu_{r_2}$. Lemma \ref{lem:diff:ind} asserts that $\lambda^i_{r_j}$ is a monotone sequence for both $j=1,2$; therefore, for such $r_1, r_2$, we get $\lambda^i_{r_1} - \lambda^i_{r_2} \leq \lambda_{r_1} - \lambda_{r_2}$, so we get
    \begin{equation}
        \label{equ:diff:right}
        \lambda^i_{r_1} - \lambda^i_{r_2} \leq s (r_2 - r_1).
    \end{equation}
    
    As $q(\lambda^i, \lambda^{i+1}) = 1$, we already observed that the Young diagrams of $\lambda^i, \lambda^{i+1}$ have only one box in different positions. This means that there exist exactly two columns $c_{i,1} < c_{i,2}$ which are not equal between $\lambda^i, \lambda^{i+1}$. Similarly, there exist exactly two rows indices $r_{i,1} < r_{i,2}$.
    
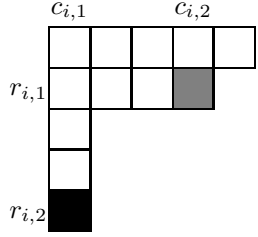
\[
        \begin{ytableau}
            \none & \none[c_{i,1}] & \none & \none & \none[c_{i,2}]\\
            \none & & & & & \\
            \none[r_{i,1}] & & & & *(gray) \\
            \none & \\
            \none & \\
            \none[r_{i,2}] & *(black) \\
        \end{ytableau}
    \]

    We have $c(\lambda^i, \lambda^{i+1}) = c_{i,2} - c_{i,1}$ and $r(\lambda^i, \lambda^{i+1}) = r_{i,2} - r_{i,1}$.
    Moreover, by definition of $\lambda^i$ and $\lambda^{i+1}$, we have $\lambda^i_{r_{i,1}} = c_{i,2}$ and $\lambda^i_{r_{i,2}} = c_{i,1} - 1$.
    Therefore, by (\ref{equ:diff:right}), we get
    \begin{align*}
        & s \cdot r(\lambda^i, \lambda^{i+1}) = s(r_{i,2} - r_{i,1}) \geq
        \lambda^i_{r_{i,1}} - \lambda^i_{r_{i,2}} = c_{i,2} - c_{i,1} + 1 = c(\lambda^i, \lambda^{i+1}) + 1.
    \end{align*}

    Finally, we consider the additional hypothesis of the existence of a column index $c$ such that $\widehat \mu_c > \widehat \lambda_c + 1$ or a row index $r$ such that $\mu_r >  \lambda_r + 1$. We are going to prove that there exists $i$ such that (\ref{equ:diff:right}) holds strictly.
    
    Let
    \[
        I_{\text{cols}}(c) = \left\{ i\in \{0, \dots, q-1\}: \widehat{\lambda^{i+1}_c} > \widehat{\lambda^i_c}\right\}.
    \]
    As $q(\lambda^i, \lambda^{i+1}) = 1$, we have $|\widehat{\lambda^{i+1}}_c - \widehat{\lambda^i}_c| \leq 1$ for each $i$. Therefore, if there exists $c$ such that $\widehat \mu_c > \widehat \lambda_c + 1$, we have $|I_{\text{cols}}(c)| \geq 2$. Let $j$ be the minimum in $I_{\text{cols}}(c)$. We show that (\ref{equ:diff:right}) holds strictly for the rows $r_{i,1}, r_{i,2}$ for every $i \in I_{\text{cols}}(c)$ such that $i \neq j$. Indeed, as $\lambda^i_{r_{i,2}} = c = \lambda^j_{r_{j,2}}$, we have
    \begin{align*}
        & \lambda^i_{r_{i,1}} - \lambda^i_{r_{i,2}} \leq
        \lambda^j_{r_{i,1}} - \lambda^j_{r_{j,2}} \leq
        \lambda_{r_{i,1}} - \lambda_{r_{j,2}} \leq \\
        & \leq
        s(r_{j,2} - r_{i,1}) \leq
        s(r_{i,2} - 1 - r_{i,1}) <
        s(r_{i,2} - r_{i,1}).
    \end{align*}
    
    In a similar fashion, let
    \[
        I_{\text{rows}}(r) = \left\{ i\in \{0, \dots, q-1\}: {\lambda^{i+1}_r} > {\lambda^i_r}\right\}.
    \]
    If there exists $r$ such that $\mu_r > \lambda_r + 1$, then $|I_{\text{rows}}(r)| \geq 2$ and we take $j \in I_{\text{rows}}(r)$ to be the minimum and $i \in I_{\text{rows}}(r)$ to be another index. We have $\lambda^i_{r_{i,2}} > \lambda^j_{r_{i,2}} = \lambda_{r_{i,2}}$, so we get
    \[
        \lambda^i_{r_{i,1}} - \lambda^i_{r_{i,2}} <
        \lambda_{r_{i,1}} - \lambda_{r_{i,2}} \leq
        s(r_{i,2} - r_{i,1}).
    \]

    In both cases (\ref{equ:diff:right}) holds strictly for at least an $i$. Therefore (\ref{equ:diff:onestep}) holds strictly for this $i$, and (\ref{equ:diff:sstep}) must hold strictly as a consequence. Thus, the lemma is proven.
\Qed\end{proof}

\begin{example}
    Let
    \[
        \lambda = (6, 4, 2), \quad \mu = (5, 3, 2, 1, 1), \quad \nu = (5, 3, 3, 1)
    \]
    be three partitions of $n=12$. We have that $\mu < \lambda$ and $\nu < \lambda$; moreover, in the pair $(\lambda, \mu)$ the first column differs by $2$ boxes, while no column in $(\lambda, \nu)$ differs more than $1$ box.
    
    Since $\lambda$ satisfies the $2$-step condition, we can compute each term involved in (\ref{equ:diff:sstep}). For $(\lambda, \mu)$ we have:
    $
        2\cdot 6 \geq 8 + 2
    $
    (which holds strictly), while for $(\lambda, \nu)$ we have:
    $
        2 \cdot 4 \geq 6 + 2
    $.
\end{example}

\section{Inequalities on the dimensions}
\label{sec:comb}

In this section we will use the tools introduced in \cref{sec:diff} to effectively compute the dimensions of the strata in $Z$.

Let $\lambda = (\lambda_1,\dots,\lambda_k)$ be a partition of $n:=|\lambda|$ with $t:=\lambda_1$ columns. We want to compare the dimension of $Z_{\tau^0}$ with the dimension of each other stratum $Z_\tau$ with $\tau \in \Lambda^{(\lambda)}$. Let $\mu = \pi(\tau)$, so that $\mu \leq \lambda$. The aim of this section is to give some sufficient combinatorial conditions on the pair of partitions $(\lambda, \mu)$ in order to secure a bound on the difference $\dim Z_{\tau^0} - \dim Z_\tau$.

The aim of the current section is to prove \cref{prop:comb:big} and \cref{prop:comb:bigr}. In order to introduce the stronger inequality in \cref{prop:comb:bigr}, we need an in-depth study of the combinatorics of $\Lambda$, so we prefer to introduce it later.
\begin{proposition}
    \label{prop:comb:big}
    Let $\lambda \geq \mu$ be two partitions and let $Z=Z^{(\lambda)}$ be the variety built from $\lambda$.  Let $Z_{\tau^0}$ be the only stratum of $Z$ with $\pi(\tau^0_1) = \lambda$ and let $Z_\tau$ be a stratum with $\tau=(\tau_1, \dots, \tau_t)$ such that $\pi(\tau_1) = \mu$.
    
    For each real number $x$, we have that:
    \[
        2r(\lambda, \mu) - c(\lambda, \mu) - q(\lambda, \mu) \geq 4x
        \quad \Longrightarrow \quad
        \dim Z_{\tau^0} - \dim Z_\tau \geq  x.
    \]
\end{proposition}

The proof will proceed as follows: we take the dimension formula of the strata of $Z$ given by (\ref{equ:zvar:dimz}) then we will estimate each term on the difference. Each of the following lemmas deals with one of the terms.

\begin{lemma}
    \label{lem:comb:col}
    Let $\lambda \geq \mu$ be two partitions of $n$. Let $t = \lambda_1$ be the number of columns of $\lambda$.  Then
    \[
        \sum_{i=1}^t \left(\widehat{\mu}_i^2 - \widehat{\lambda}_i^2\right) = 2r(\lambda, \mu).
    \]
    (We allow $\widehat{\mu}_i = 0$ for each $i > \mu_1$.)
\end{lemma}
\begin{proof}
    For every integer $m$ we have: $m^2 = 2\binom {m+1} 2 - m$. So, by (\ref{equ:diff:trivsum}), we get:
    \[
        \sum_{i=1}^t \left(\widehat{\mu}_i^2 - \widehat{\lambda}_i^2\right) =
        \sum_{i=1}^t \left(2\binom{\widehat{\mu}_i+1} 2 - \widehat{\mu_i} - 2\binom{\widehat{\lambda}_i+1} 2 + \widehat{\lambda}_i\right) =
        2r(\lambda, \mu) - n + n.\Qed
    \]
\end{proof}

In order to carry on the computation, we examine carefully all the strata $Z_{\tau}$ with $\tau \neq \tau^0$. We do so by grouping together the $\tau \in \Lambda$ which can appear in $\Theta^{-1}(C_{\mu})$, for fixed $\mu$. We already know that for every $\tau = (\tau_1, \dots, \tau_t)$ such that $\Theta(Z_\tau) = C_\mu$, we have that $\pi(\tau_1) = \mu$.

The conditions $\tau \in \Lambda$ and $\pi(\tau_1) = \mu$ place important combinatorial restrictions in the choice of the $ab$-diagrams $\tau_1, \dots, \tau_t$. We are going to describe them in the following paragraphs.

Let $Z^{(\mu)}$ be the variety $Z$ built from the partition $\mu$ (rather than the partition $\lambda$). We will denote by $\sigma^0$ the unique stratum in $\Lambda^{(\mu)}$ such that $\Theta(Z^{(\mu)}_{\sigma^0}) = C_\mu$. As in the case of $\tau^0$, every $ab$-diagram $\sigma^0_i$ has only rows starting and ending by $a$, in particular they all have odd length, so:
\begin{equation}
    \label{equ:comb:ots}
    \sum_{i=1}^t o(\sigma^0_i) = n = \sum_{i=1}^t o(\tau^0_i).
\end{equation}
The sequence of $ab$-diagrams $\sigma^0$ does not belong to $\Lambda^{(\lambda)}$ because the integers $n_i$ computed from $\mu$ are different from those computed from $\lambda$.

Nonetheless, for every $\tau \in \Lambda^{(\lambda)}$ such that $\pi(\tau_1) = \mu$, the condition $\rho(\tau_{i-1}) = \pi(\tau_i)$ implies that each $\tau_i$ can be obtained from $\sigma^0_i$ by adding an adequate amount of $a$'s and $b$'s letters. For example, if $i=0$ we need some $b$'s, possibly.

For each $i = 1, \dots, t$, let $d_i$ be the list of $a$'s and $b$'s we need to add to $\sigma^0_i$ in order to obtain $\tau_i$. Let $d_i^a$ (resp. $d_i^b$) be the number of $a$'s (resp. $b$'s) in $d_i$. We remark that $d_i^a$ and $d_i^b$ depend only on $(\lambda, \mu)$ and not on the particular $\tau$. In fact, we can compute $d_i^a$ by taking the differences
\begin{equation}
    \label{equ:comb:dlist}
    d_i^a = \sum_{j=i}^t \left(\widehat{\lambda}_j - \widehat{\mu}_j\right);
\end{equation}
and therefore we have $d_i^b = d_{i+1}^a$.

\begin{example}
    Let $\tau \in \Lambda^{(\lambda)}$ and let $\sigma^0 \in \Lambda^{(\mu)}$. Suppose that the $i$-th diagrams are the following:
    \[
        \sigma^0_i =
        \begin{array}{l}
                aba \\
                aba \\
                b
            \end{array};
    \qquad
    \tau_i =
            \begin{array}{l}
                ababa \\
                aba \\
                ba \\
                ab
            \end{array}.
    \]
    Then, we have $d_i = (a, a, a, b, b)$, $d_i^a = 3$, $d_i^b = 2$.
\end{example}

Given an $ab$-diagram $\delta$ and a list $d$ of $a$'s and $b$'s, we define $\aug_\delta(d)$ to be the set of ortho-symmetric $ab$-diagrams obtainable from $\delta$ by adding the letters in $d$.
\begin{example}
    Let $d = (a, b)$ and
    \[
    \delta =
        \begin{array}{l}
            aba \\
            aba \\
            b
        \end{array}.
    \]
    Then $\aug_\delta(d)$ has three elements:
    \[
        \aug_\delta(d) = \left\{
        \begin{array}{l}
            ababa \\
            aba \\
            b
        \end{array},\quad
        \begin{array}{l}
            aba \\
            aba \\
            bab
        \end{array},\quad
        \begin{array}{l}
            aba \\
            aba \\
            a \\
            b \\
            b
        \end{array}\right\}.
    \]
\end{example}

\begin{lemma}
    \label{lem:comb:maxab}
    Let $\delta^0$ be an $ab$-diagram with rows starting and ending only with $a$. Let $d$ be a list of $d^a$ $a$'s and $d^b$ $b$'s. Let $\delta \in \aug_{\delta^0}(d)$. Then the following holds:
    \[
        o(\delta) - 2\Delta(\delta) - o(\delta^0) \leq \max\{d^a, d^b\}.
    \]
\end{lemma}
\begin{proof}
    The $ab$-diagram $\delta$ is obtained from $\delta^0$ by adding $d^a$ $a$'s and $d^b$ $b$'s.
    
    Let $L$ (resp. $S$) be the set of rows of $\delta$ longer than $1$ (resp. of length $1$) built using only letters in $d$. As $o(\delta)$ counts some of the rows in $\delta$, we compare the difference $o(\delta) - o(\delta^0)$ with the number of rows built only with letters in $d$ and we get $o(\delta) - o(\delta^0) \leq |L| + |S|$.
    
    Let $S_a$ (resp. $S_b$) the rows of $S$ starting with $a$ (resp. $b$), so that $S = S_a \sqcup S_b$. Recall that $a_i$ (resp. $b_i$) is the number of rows of length $i$ starting with $a$ (resp. $b$); in this particular $ab$-diagram $\delta$ we have $|S_b| = b_1$ and $|S_a| \leq a_1$.
    
    \begin{align*}
        o(\delta) - o(\delta^0) - 2\Delta(\delta) \leq & |L| + |S| -2\sum_{i \text{ odd}} a_i b_i\\
        \leq & |L| + |S| -2 a_1 b_1 \\
        \leq & |L| + |S_a| + |S_b| - 2 |S_a| |S_b|\\
        \leq & |L| + \max\{|S_a|, |S_b|\}.
    \end{align*}
    If $|S_a| \geq |S_b|$, then $|L| + |S_a| \leq d_i^a$ because each row in $L$ or in $S_a$ contains at least one $a$. If the opposite holds true, then $|L| + |S_b| \leq d_i^b$ for the same reason. Thus the lemma is proven.\Qed
\end{proof}

We can introduce a small, but very important, improvement on \cref{lem:comb:maxab}.

\begin{lemma}
    \label{lem:comb:maxab2}
    In the same setting of \cref{lem:comb:maxab}, if we furthermore assume that $d = (b)$, that is $d^a = 0$ and $d^b = 1$, and that $\delta^0$ has $l$ rows of length $1$, we have the stronger equality:
    \[
        o(\delta) - 2\Delta(\delta) - o(\delta^0) = \max\{d^a, d^b\} - 2l = 1 - 2l.
    \]
\end{lemma}
\begin{proof}
    As all the rows in $\delta^0$ have odd length, we cannot extend exactly one row with letter $b$, otherwise the resulting $ab$-diagram would not be ortho-symmetric. So we must place $b$ on a new row.
    
    In this case, $o(\delta) - o(\delta^0) = 1$ and $\Delta(\delta) = a_1 b_1 = l\cdot 1$.
\Qed\end{proof}

Using this lemma we will be able to introduce the previously announced stronger version of \cref{prop:comb:big}, namely:
\begin{proposition}
    \label{prop:comb:bigr}
    Let $\lambda \geq \mu$ be two partitions and let $Z=Z^{(\lambda)}$ be the variety built from $\lambda$.  Let $Z_{\tau^0}$ be the only stratum of $Z$ with $\pi(\tau^0_1) = \lambda$ and let $Z_\tau$ be a stratum with $\tau=(\tau_1, \dots, \tau_t)$ such that $\pi(\tau_1) = \mu$.
    As in \cref{prop:comb:big}, we suppose that there exists a real number $x$ such that:
    \[
        2r(\lambda, \mu) - c(\lambda, \mu) - q(\lambda, \mu) \geq 4x.
    \]

    If there exists an index $i$ such that $d_i = (b)$ and $\sigma^0_i$ has $l$ rows of length one, then:
    \[
        \dim Z_{\tau^0} \geq \dim Z_\tau + x + l/2.
    \]
\end{proposition}
The proof of \cref{prop:comb:bigr} will be given together with the proof of \cref{prop:comb:big} at the end of this section.

We want an estimate of the sum of the terms $\max\{d_i^a, d_i^b\}$ depending only on $\lambda, \mu$.

\begin{lemma}
    \label{lem:comb:clem}
    Let $\lambda \geq \mu$ be partitions, so that we can define $d_i^a, d_i^b$ for all $i = 1, \dots, t$. Then
    \begin{equation}
        \label{equ:comb:clem}
        \sum_{i=1}^{t} \max\{d_i^a, d_i^b\} \leq c(\lambda, \mu) + q(\lambda, \mu).
    \end{equation}
\end{lemma}
\begin{proof}
    First we prove the claim for $q(\lambda, \mu) = 1$.
    
    In this case, $\mu$ is obtained from $\lambda$ by moving down a single box $B$. Let us call $c_1$ (resp. $c_2$) the column where $B$ lies in $\mu$ (resp. $\lambda$); so $c_1 < c_2$. By (\ref{equ:comb:dlist}), one gets $d_i^a = 1$ for each $c_1 < i \leq c_2$, and $0$ otherwise. Moreover one also gets that $d_i^b = 1$ for each $c_1 \leq i < c_2$ and $0$ otherwise. Therefore $\max\{d_i^a, d_i^b\} = 1$ for all $c_1 \leq i \leq c_2$ and $0$ otherwise. Therefore
    \[
        \sum_{i=1}^t \max\{d_i^a, d_i^b\} = c_2 - c_1 + 1,
    \]
    while
    \[
        c(\lambda, \mu) = c_2 - c_1,
    \]
    so
    \begin{equation}
        \label{equ:comb:qone}
        \sum_{i=1}^t \max\{d_i^a, d_i^b\} = c(\lambda, \mu) + 1
    \end{equation}
    and the conclusion holds in this case.
    
    In the general case, let $q:=q(\lambda, \mu)$ so, by \cref{lem:diff:ind}, we get a sequence of partitions
    \[
        \lambda = \lambda^0 > \lambda^1 > \dots > \lambda^{q} = \mu
    \]
    such that $q(\lambda^j, \lambda^{j+1}) = 1$ for each $j = 0, \dots, q-1$. By (\ref{equ:comb:qone}), for every $j$, we have
    \[
        \sum_{i=1}^t \max\{d_i^a, d_i^b\}(\lambda^j, \lambda^{j+1}) = c(\lambda^j, \lambda^{j+1}) + 1
    \]
    and, summing up over $j$,
    \[
        \sum_{i=1}^t \sum_{j=0}^{q-1}\max\{d_i^a, d_i^b\}(\lambda^j, \lambda^{j+1}) = \sum_{j=0}^{q-1}c(\lambda^j, \lambda^{j+1}) + q.
    \]
    
    By (\ref{equ:diff:cadd}), the function $c$ is additive on the pairs ($\lambda^j, \lambda^{j+1}$), therefore the right hand side is equal to $c(\lambda, \mu) + q$. Also $d_i^a$ and $d_i^b$ are additive functions on the pairs $(\lambda^j, \lambda^{j+1})$ (as it follows from (\ref{equ:comb:dlist})), therefore
    \begin{align*}
        \sum_{j=0}^{q-1}\max\{d_i^a, d_i^b\}(\lambda^j, \lambda^{j+1}) 
        & = \sum_{j=0}^{q-1}\max\left\{d_i^a(\lambda^j, \lambda^{j+1}), d_i^b(\lambda^j, \lambda^{j+1})\right\}\\
        & \geq \max\left\{\sum_{j=0}^{q-1}d_i^a(\lambda^j, \lambda^{j+1}), \sum_{j=0}^{q-1}d_i^b(\lambda^j, \lambda^{j+1})\right\} \\
        & = \max\left\{d_i^a(\lambda, \mu), d_i^b(\lambda, \mu)\right\} = \max\{d_i^a, d_i^b\}
    \end{align*}
    and the conclusion follows.
\Qed\end{proof}

We are finally able to prove both \cref{prop:comb:big} and \cref{prop:comb:bigr}.
\begin{proof}[Proof of \cref{prop:comb:big} and \cref{prop:comb:bigr}]
    We start from the dimension formula of the strata (\ref{equ:zvar:dimz}). We recall that $\Delta(\tau^0) = 0$ as no $ab$-diagram of $\tau^0$ has rows starting with $b$. Therefore
    \begingroup
    \allowdisplaybreaks
    \begin{align*}
         4\big(\dim Z_{\tau^0} - \dim Z_\tau - x\big) = &\ 2 (\dim C_{\pi(\tau_1^0)} - \dim C_{\pi(\tau_1)})+\\
         & + \left(o(\tau^0) - 2\Delta(\tau^0)\right) - \left(o(\tau) - 2\Delta(\tau)\right) - 4x\\
         = & \ 2 (\dim C_{\lambda} - \dim C_{\mu}) + o(\tau^0) - o(\tau) + 2\Delta(\tau) - 4x\\
         = & \sum_{i=1}^{t} \left(\widehat{\mu}_i^2-\widehat{\lambda}_i^2\right) + \sum_{i=1}^t \left(o(\tau^0_i) - o(\tau_i) + 2\Delta(\tau_i)\right) - 4x\\
        \text{(by \cref{lem:comb:col})} = &\ 2r(\lambda, \mu) + \sum_{i=1}^t \left(o(\tau^0_i) - o(\tau_i) + 2\Delta(\tau_i)\right) - 4x\\
        \text{(by (\ref{equ:comb:ots}))} = &\ 2r(\lambda, \mu) + \sum_{i=1}^t \left(o(\sigma^0_i) - o(\tau_i) + 2\Delta(\tau_i)\right) - 4x\\
        \geq &\ 2r(\lambda, \mu) - 4x\ + \\
        & \ -\sum_{i=1}^t \max_{\tau_i \in \aug_{\sigma^0_i}(d_i)} \big(o(\tau_i) - 2\Delta(\tau_i) - o(\sigma^0_i) \big)\\
        \text{(by \cref{lem:comb:maxab})} \geq &\ 2r(\lambda, \mu) -  \sum_{i=1}^t \max\{d_i^a, d_i^b\} - 4x\\
        \text{(by \cref{lem:comb:clem})} \geq &\ 2r(\lambda, \mu) - c(\lambda, \mu) - q(\lambda, \mu) - 4x.
    \end{align*}
    \endgroup
    
    Therefore $\dim Z_{\tau^0} - \dim Z_\tau \geq x$ is guaranteed as soon as $2r(\lambda, \mu) - c(\lambda, \mu) - q(\lambda, \mu) \geq 4x$.
    
    In order to obtain the sharper result of \cref{prop:comb:bigr}, it is enough to use  \cref{lem:comb:maxab2} in place of \cref{lem:comb:maxab} in the second to last step.
\Qed\end{proof}

\section{Complete intersection conditions}
\label{sec:ci}
In this section we want to give condition under which we can make sure that $Z$ is a complete intersection variety.

\begin{proposition}
    \label{prop:ci:codim}
    Let $\lambda$ be a partition and recall that $Z = \Phi^{-1}(0) \subseteq M$. We have
    \[
        \codim_M(Z_{\tau^0}) = \dim N.
    \]
\end{proposition}
\begin{proof}
    In order to prove the equality of the required dimensions, we will work with the dimension formula given by (\ref{equ:zvar:dimz}). We are going to prove that
    \begin{equation}
        \label{equ:ci:codim}
        \dim Z_{\tau^0} = \dim M - \dim N.
    \end{equation}
    
    We immediately have:
    \[
        \dim M = \sum_{i=1}^t \dim L(V_{i-1},V_i) = \sum_{i=1}^t n_{i-1}n_i
    \]
    and
    \[
        \dim N = \sum_{i=1}^{t-1} \dim \p(V_i) = \sum_{i=1}^{t-1}\frac 1 2 n_i^2 + \sum_{i=1}^{t-1}\frac 1 2 n_i,
    \]
    so
    \[
        \dim M - \dim N = \sum_{i=1}^t n_{i-1}n_i - \frac 1 2\sum_{i=1}^{t-1} n_i^2 - \frac 1 2\sum_{i=1}^{t-1} n_i.
    \]

    On the other hand, we have:
    \[
        \begin{split}
            \dim Z_{\tau^0} =
            \frac 1 2 \dim C_{\pi(\tau^0_1)} + \sum_{i=0}^{t-1}\left(\frac 1 2 n_in_{i+1} - \frac 1 4 (n_i + n_{i+1})\right) + \\
            +\sum_{i=1}^t \left(\frac 1 4  o(\tau^0_i) -\frac 1 2\Delta(\tau^0_i)\right).
        \end{split}
    \]
    
    We have that $\pi(\tau^0_1) = \lambda$; we then use formula (\ref{equ:sno:dim}).
    
    We also have that $o(\tau^0_i)$ is precisely the number of rows of $\tau^0_i$ (as they all have odd length), the number of rows of $\tau^0_i$ is equal to the number of rows of $\lambda^i :=\pi(\tau^0_i)$, the partition $\lambda^i$ has clearly $\widehat{\lambda}^i_1$ rows and $\widehat{\lambda}^i_1 = \widehat{\lambda}_i$, that is the $i$-th column of $\lambda$.
    
    Finally, we have that $\Delta(\tau^0_i) = 0$, because there are no rows in $\tau^0_i$ starting with $b$.
    
    Therefore:
    \[
        \dim Z_{\tau^0} = \frac 1 4 \left(n_0^2- \sum_{i=1}^t \widehat\lambda_i^2\right) + \sum_{i=0}^{t-1}\left(\frac 1 2 n_in_{i+1} - \frac 1 4 (n_i + n_{i+1})\right) + \frac 1 4 \sum_{i=1}^t\widehat{\lambda}_i;
    \]
    we recall that $\widehat\lambda_i = n_{i-1} - n_i$ and that $n_0 = n = |\lambda|$, $n_t=0$; so
    \begin{align*}
        \dim Z_{\tau^0} =& \frac 1 4 \left(n_0^2- \sum_{i=1}^t (n_{i-1}-n_i)^2\right) +\\
        &+ \frac 1 2\sum_{i=0}^{t-1} n_in_{i+1} - \frac 1 4 n_0 - \frac 1 2\sum_{i=1}^{t-1}n_i + \frac 1 4 \sum_{i=1}^tn_{i-1}-n_i\\
        =&- \frac 1 2\sum_{i=1}^{t-1} n_i^2 + \frac 1 2\sum_{i=1}^{t-1} n_{i-1}n_i + \frac 1 2 \sum_{i=0}^{t-1}n_in_{i+1} - \frac 1 2 \sum_{i=1}^{t-1} n_i\\
        =& \dim M - \dim N.\Qed
    \end{align*}
\end{proof}

We can proceed to the main result of this section.

\begin{proposition}
    \label{pro:ci:ci}
    If $\lambda$ satisfies the $2$-step condition, then $Z$ is a complete intersection variety.
\end{proposition}

The argument is basically the same used in \cite{kraft1979closures}*{Theorem in 3.3} and \cite{kraft1982geometry}*{Theorem in 5.3}. Here we merely put the pieces together. One of these pieces is the following technical lemma.
\begin{lemma}
    \label{lem:ci:majineq}
    If the partition $\lambda$ satisfies the $2$-step condition and $\tau \in \Lambda$ is a string of $ab$-diagrams different from $\tau^0$, then
    \[
        \dim Z_{\tau^0} > \dim Z_\tau.
    \]
\end{lemma}

\begin{proof}[Proof of \cref{pro:ci:ci}]
    We start recalling \cref{prop:zvar:zsmooth} and \cref{prop:ci:codim}.
    
    As $Z \setminus Z_{\tau^0}$ consists of finitely many strata $Z_\tau$ and, by \cref{lem:ci:majineq},  $\codim_Z(Z_\tau) \geq 1$, we deduce that $Z = \overline{Z_{\tau^0}}$. That implies that $Z$ is a complete intersection variety smooth in codimension 0 and $Z$ is reduced as a scheme.
\Qed\end{proof}

We can finally use \cref{lem:diff:usef} and propositions \ref{prop:comb:big} and \ref{prop:comb:bigr} to prove \cref{lem:ci:majineq}.
\begin{proof}[Proof of \cref{lem:ci:majineq}]
    Let $\mu = \pi(\tau)$, so in particular $\mu < \lambda$, as $\tau \neq \tau^0$. Therefore \cref{lem:diff:usef} implies
    \begin{equation}
        \label{equ:ci:rcqineq}
        2 r(\lambda, \mu) \geq c(\lambda, \mu) + q(\lambda, \mu)
    \end{equation}
    and, by combining it with \cref{prop:comb:big} we immediately get
    \[
        \dim Z_{\tau^0} \geq \dim Z_\tau.
    \]
    So our concern is reduced to obtain a strict inequality. We will gain the strict inequality either by proving that (\ref{equ:ci:rcqineq}) holds strictly for the pair $(\lambda, \mu)$ or by using \cref{prop:comb:bigr} with $l > 0$.
    
    Let $c$ be the first column such that $\widehat\mu_c > \widehat\lambda_c$ and let $r = \widehat\mu_c$. This means that the first $c-1$ columns of $\lambda$ and $\mu$ are the same and $\mu_r > \lambda_r$. On the basis of $c$ and $r$, we distinguish three cases.
    
    \begin{description}
        \item[$\widehat\mu_c - \widehat\lambda_c \geq 2$]: in this case we can use \cref{lem:diff:usef} with the column $c$ to obtain that
        \[
            2r(\lambda, \mu) > c(\lambda, \mu) + q(\lambda, \mu).
        \]
        \item[$\widehat\mu_c - \widehat\lambda_c = 1$, $\mu_r - \lambda_r \geq 2$]: we can still use \cref{lem:diff:usef}, this time with the row $r$, to obtain that
        \[
            2r(\lambda, \mu) > c(\lambda, \mu) + q(\lambda, \mu).
        \]
        \item[$\widehat\mu_c - \widehat\lambda_c = 1$, $\mu_r - \lambda_r = 1$]: in this case we can use \cref{prop:comb:bigr} with $i = c$ and $l \geq 1$. Indeed, the following two facts are immediately seen: by (\ref{equ:comb:dlist}), $d_c = (b)$; the $ab$-diagram $\sigma^0_c$ has at least one row equal to a single $a$, namely the $\widehat\mu_c$-th row.\Qed
    \end{description}
\end{proof}

\section{Normality conditions}
\label{sec:nor}
In this section we prove \cref{thm:intro:main}. We start by proving the necessary condition as a consequence of the works by Ohta and Sekiguchi.
\begin{theorem}[the \emph{only if} part of \cref{thm:intro:main}]
    \label{thm:nor:nec}
    Let $\lambda$ be a partition of $n$ and suppose that $\lambda$ does not satisfy the $1$-step condition. Then $\overline{C_\lambda} \subseteq \p$ is not normal.
\end{theorem}
\begin{proof}
    If $\lambda$ does not satisfy the $1$-step condition, then there is an $i$ such that $\lambda_{i} \geq \lambda_{i+1} + 2$. Let
    \[
        \mu = (\lambda_1, \dots, \lambda_{i-1}, \lambda_i - 1, \lambda_{i+1} + 1, \dots, \lambda_h)
    \]
    be another partition of $n$, so that $\mu < \lambda$ and $\mu$ is a minimal degeneration of $\lambda$ (in the same sense of \citelist{\cite{kraft1980minimal}\cite{ohta1986singularities}}). Let $m = \lambda_i - \lambda_{i+1} \geq 2$ and let us define the following two partitions of $m$:
    \begin{align*}
        &\lambda' = (m), \\
        &\mu' = (m-1, 1).
    \end{align*}

    The degeneration $\mu' < \lambda'$ is obtained from $\mu < \lambda$ by removing the common rows and columns (of $\mu$ and $\lambda$). By theorem \cite{ohta1986singularities}*{Theorem 2}, the singularity $C_\mu$ of $\overline{C_\lambda}$ is normal if and only if the singularity $C_{\mu'}$ of $\overline{C_{\lambda'}}$ is normal.

    Moreover, by construction, $C_{\lambda'}$ is the variety of the principal (or regular) nilpotent elements of $\p$ (in dimension $m$); therefore we also have that $\overline{C_{\lambda'}} = \mathcal N(\p)$, that is the cone of the nilpotent elements in $\p$. As $\mu'$ is the minimal degeneration of $\lambda'$, $C_{\mu'}$ is the variety of subregular nilpotent elements of $\p$. The theorem \cite{sekiguchi1984nilpotent}*{\S 3.1, Theorem 7.} proves that the singularity $C_{\mu'}$ is not normal in $\mathcal N(\p)$. So this concludes the \emph{only if} part.
\Qed\end{proof}

We turn to the sufficient condition. We will make use of the construction of the variety $Z$.

\begin{theorem}[the \emph{if} part of \cref{thm:intro:main}]
    \label{thm:nor:suff}
    If the partition $\lambda$ satisfies the $1$-step condition, then $Z$ is a normal variety. In particular $\overline{C_\lambda}$ is a normal variety.
\end{theorem}

As $\lambda$ satisfies the $1$-step condition, also $\lambda$ satisfies the $2$-step condition, so, by \cref{pro:ci:ci}, $Z$ is a complete intersection variety. At this point the proof of \cref{thm:nor:suff} is a direct consequence of the following lemma, as in \cite{kraft1979closures}*{Section 3.7}.

\begin{lemma}
    If the partition $\lambda$ satisfies the $1$-step condition, then $\dim (Z \setminus Z_{\tau^0}) \leq \dim Z - 2$.
\end{lemma}
\begin{proof}
    Let $\tau \in \Lambda$ such that $\tau \neq \tau^0$. Thus $\mu := \pi(\tau_1) < \lambda$. We need to prove that
    \[
         \dim Z_{\tau^0} - \dim Z_{\tau} \geq 2.
    \]
    By \cref{lem:diff:usef}, $r(\lambda, \mu) \geq c(\lambda, \mu) + q(\lambda, \mu)$ and hence
    \[
        2r(\lambda, \mu) - c(\lambda, \mu) - q(\lambda, \mu) \geq 4x
    \]
    where $x = \frac 14\left(c(\lambda, \mu) + q(\lambda, \mu)\right)$. By \cref{prop:comb:big}, we get
    \[
        \dim Z_{\tau^0} - \dim Z_\tau \geq x,
    \]
    so we are done if $x > 1$.
    
    Let us assume that $x \leq 1$. For simplicity of notation, we put $c:=c(\lambda, \mu)$ and $q:=q(\lambda, \mu)$. We claim that there are only four cases left: $q = c = 2$ and $q = 1$, $c \leq 3$. Indeed, by (\ref{equ:diff:equcases}), $\lambda > \mu$ if and only if both $c \geq 1$ and $q \geq 1$. By (\ref{equ:diff:ineq}), we also have $c \geq q$. Since $4x = c+q \leq 4$, then $q < 3$. Indeed, $q \geq 3 \Rightarrow c \geq 3 \Rightarrow c+q \geq 6$. If $q = 2$, then $c \leq 2$ by the same reasoning. If $q = 1$, then $c = 4x-q \leq 4 - 1$.
    
    We analize the four cases separately and, in each case, we find an index $i$ such that $d_i = (b)$, the $ab$-diagram $\sigma^0_i$ has $l$ rows of length one and $l$ is big enough so that we can conclude the proof by \cref{prop:comb:bigr} as follows:
    \[
        \dim Z_{\tau^0} - \dim Z_\tau \geq x + \frac l2 = \frac 14(c+q) + \frac l2 > 1.
    \]
    

    \begin{description}
        \item[$q = 2$, $c = 2$] There are two boxes $B$, $B'$ which are lowered from $\lambda$ in order to obtain $\mu$. Let $i'$ be the column of $B$ in $\lambda$ and $i$ be the column of $B$ in $\mu$. As $c = 2$, $i = i'-1$. By (\ref{equ:comb:dlist}), the list $d_i = (b)$. In $\sigma^0_i$  the rows with indices $\widehat \mu_i = \widehat \lambda_i + 1$, $\widehat \lambda_i$ and $\widehat \lambda_i-1$ have length one. Indeed, since $\lambda$ is $1$-step, $\lambda$ does not have two different columns with the same height. Thus, $\sigma^0_i$ has $l \geq 3$ rows of length one. We found that $\frac 14(c+q) + \frac l2 \geq \frac 52 > 1$.
        
        \item[$q = 1$, $c = 3$] There is a single box $B$ which is lowered from $\lambda$ in order to obtain $\mu$. Let $i'$ be the column of $B$ in $\lambda$ and $i$ be the column of $B$ in $\mu$. As $c = 3$, $i = i'-3$. By (\ref{equ:comb:dlist}), the list $d_i = (b)$. In $\sigma^0_i$ the the rows with indices $\widehat \mu_i$, $\widehat \mu_i-1$ have length one. Thus, $\sigma^0_i$ has $l \geq 2$ rows of length one. We found that $\frac 14(c+q) + \frac l2 \geq 2 > 1$.
        
        \item[$q = 1$, $c = 2$] Similar to the previous case, there is a single box $B$ moving from column $i'=i+2$ to column $i$; $d_i = (b)$; the rows with indices $\widehat \mu_i$, $\widehat \mu_i-1$ have length one; so $\sigma^0_i$ has $l \geq 2$ rows of length one. We found that $\frac 14(c+q) + \frac l2 \geq \frac 74 > 1$.

        \item[$q = 1$, $c = 1$] There is a single box $B$ moving from column $i'=i+1$ to column $i$. Similar to the case $q = 2$, $c = 2$, $d_i = (b)$, the rows with indices $\widehat \mu_i = \widehat \lambda_i + 1$, $\widehat \lambda_i$ and $\widehat \lambda_i-1$ have length one and $\sigma^0_i$ has $l \geq 3$ rows of length one. We found that $\frac 14(c+q) + \frac l2 \geq 2 > 1$.\Qed
    \end{description}
\end{proof}

\end{document}